\documentclass{amsart}
\usepackage[margin=1in]{geometry} 
\usepackage{lipsum}
\usepackage{tikz-cd}
\usepackage[unicode]{hyperref}
\usepackage{graphicx,color}
\usepackage{amssymb,amsmath}
\usepackage{mathtools}
\usepackage{mathrsfs}
\usepackage{enumerate}
\usepackage{ragged2e}
\usepackage{enumitem,pifont}
\usepackage{xypic}
\usepackage{tikz}
\usetikzlibrary{arrows,decorations.pathmorphing,automata,backgrounds}
\usetikzlibrary{backgrounds,positioning}

\usepackage{caption}
\usepackage{subcaption}
\usepackage{float}

\usepackage{egothic}
\usepackage[T1]{fontenc}
\usepackage{yfonts}

\graphicspath{{Figures/}}

\usepackage{calc}

\usepackage[all]{xy}
\usepackage{amsfonts}
\usepackage{amsthm}
\usepackage{mathrsfs}
\usepackage{graphicx}
\usepackage{epstopdf}
\usepackage{tgtermes}
\usepackage{yfonts}
\usepackage{cleveref}
\usepackage{stix}
\usepackage{relsize}
\usepackage{tikz}
\usetikzlibrary{matrix}
\usepackage{svg}

\newcommand{\R}{\mathbb{R}}

\newcommand{\T}{\mathbb{T}}

\newcommand{\Conn}{\operatorname{Conn}}

\newcommand{\Tube}{\operatorname{Tube}}

\newtheorem{theorem}{Theorem} [section]

\newtheorem{prop}[theorem]{Proposition}

\newtheorem{cor}[theorem]{Corollary}

\theoremstyle{definition}     
\newtheorem{definition}[theorem]{Definition}

\theoremstyle{remark}
\newtheorem{remark}[theorem]{Remark}
\newtheorem{example}[theorem]{Example}

\usepackage{multicol}

\title{A Reidemeister Theorem for Solid Ribbon Torus Links}

\author[Z. Dancso]{Zsuzsanna Dancso}
\address{School of Mathematics and Statistics\\ The University of Sydney\\ Sydney, NSW, Australia}
\email{zsuzsanna.dancso@sydney.edu.au}
\urladdr{zsuzsannadancso.net}

\author[T. Malliaras Gavrielatos]{Thomas Malliaras Gavrielatos}
\address{School of Mathematics and Statistics\\ The University of Sydney\\ Sydney, NSW, Australia}
\email{t.malliarasgavrielatos@sydney.edu.au}


\date{\today}

\begin{document}

\maketitle	
\begin{abstract}
A complete Reidemeister characterisation of welded links is a long-standing open problem. We present a Reidemeister Theorem for a related class of four-dimensional links, solid ribbon torus links: immersed solid tori in $\R^4$ with only ribbon singularities, considered up to generalised ribbon isotopy. 
\end{abstract}

\tableofcontents

\section{Introduction}
The purpose of this note is to prove a Reidemeister Theorem for {\em solid ribbon torus links}, which are close cousins of {\em welded links}.

Welded links are tori embedded in $\R^4$ which are fillable to solid tori with only {\em ribbon} singularities (details below). In the 1960's Yajima and Yanagawa \cite{Yajima, Yan1, Yan2, Yan3} studied ribbon 2-knots (which are spheres, rather than tori, embedded in $\R^4$). Welded knot theory in its current form emerged from the work of Fenn, Rimanyi, and Rourke \cite{Fenn-Rimanyi-Rourke} on welded braids. In this context a Reidemeister theorem -- or Artin presentation -- exists \cite{BH,Gol,Satoh}. The braid approach can be applied to prove a Reidemeister theorem for welded {\em homotopy} string links, where self-crossings of the same component are stransparent (virtualisable) \cite{ABMW1,ABMW2}.  However, a full Reidemeister description of {\em welded links} remains elusive.

A welded torus link is a {\em ribbon} embedding of finitely many disjoint tori in $\R^4$, as in \cref{IntroImage}. A ribbon embedding is one which admits a filling to immersed solid tori, where singularities are ribbon. A ribbon singularity is a transverse singular disc with two preimages: 
\begin{itemize}
\item one preimage is a contractible disc in the interior of a solid torus, and 
\item the other preimage's circle boundary is essential in the boundary of a solid torus. 
\end{itemize}
Welded torus links are considered up to ambient isotopy via {\em generalised ribbon} embeddings (\cref{generalised singularities for ribbons}).

\begin{figure}[t]
\centering
\includegraphics[width=4cm]{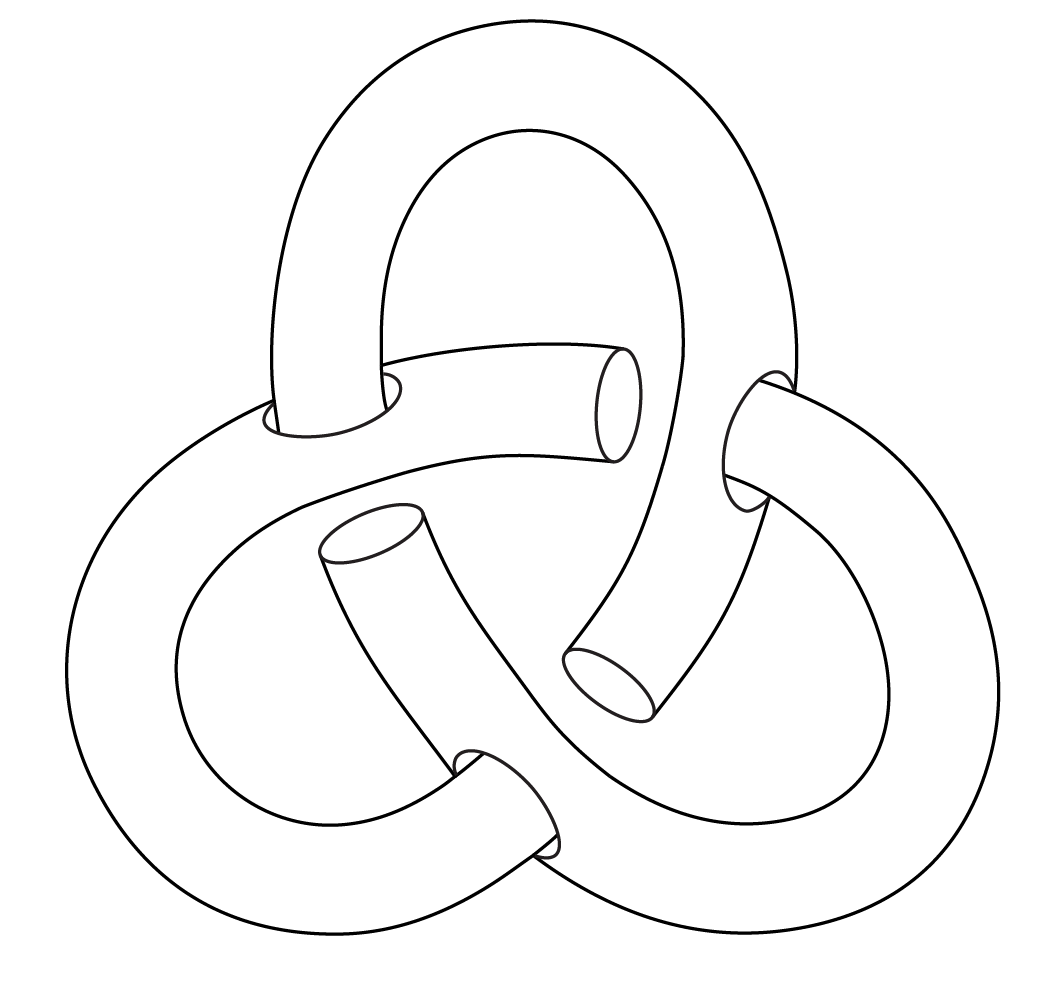}
\hspace{2cm}
\includegraphics[width=6cm]{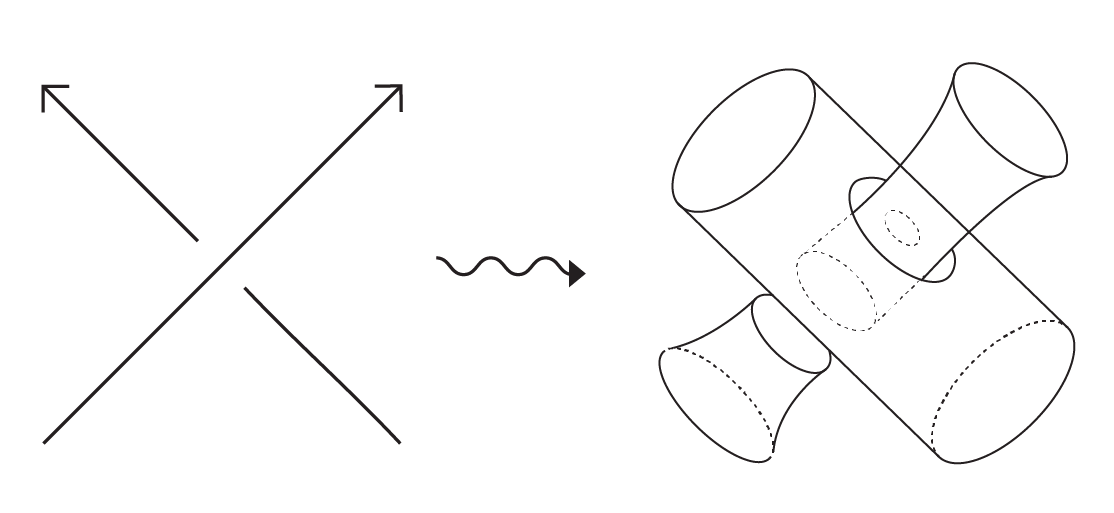}
\caption{On the left, the welded trefoil knot. On the right, the bulding block corresponding to a positive crossing. Both are shown as broken surface diagrams: the broken parts of the surface are lower in the 4th coordinate.}\label{IntroImage}
\end{figure}

A welded link diagram (\cref{def:WeldedDiag}) is a finite set of immersed circles in $\R^2$, with two types of transverse double points called {\em crossings}. Crossings can be classical (endowed with over/under information) or {\em virtual} (without over/under information). Welded link diagrams are considered up to a set of welded {\em Reidemeister moves} (\cref{Reid_moves}).

To each welded link diagram, the {\em Tube map} \cite{Yajima, Satoh, BND} associates a welded link, mapping equivalent diagrams to isotopic links. Informally, to each classical crossing the Tube map associates a ``building block'' consisting of a tube braided through another as in the right of \cref{IntroImage}; or equivalently obtained by spinning a classical crossing around a regular projection plane embedded in $\mathbb{R}^4$. The building blocks are then connected by gluing in annuli according to the diagram. Virtual crossings are ignored: they are only important in their capacity to allow non-planar connections between the ends of crossings. 

It is easy to see that the welded Reidemeister moves are in the kernel of the tube map. It is known that the kernel is at least slightly larger than this, by a global symmetry \cite{Satoh, Win}. It remains an open question to fully determine the kernel; the key ingredient missing is a combinatorial understanding of {\em filling changes}.

In this note we circumvent the filling change issue by including the filling information in the data of a link: that is, by studying {\em solid ribbon torus links}. Using Satoh's \cite{Satoh} generalisation of Yajima's  Tube map, Audoux (\cite[Prop. 3.7]{Audoux}, also see \cite{Audoux:WT}) gave a sketch of a proof for a Reidemeister theorem for solid ribbon string links up to strict ambient isotopy (no Reidemeister moves). Our goal is to prove the following welded Reidemeister theorem -- stated without proof in \cite{AM} -- which gives a full combinatorial description for solid ribbon torus links up to generalised ribbon isotopies, showing that these always correspond to welded Reidemeister moves:

\begin{theorem}\label{thm:main}
Welded link diagrams modulo welded moves are in bijection with solid ribbon torus links up to generalised ribbon isotopy. 
\end{theorem}

The proof expands on Audoux's idea \cite[Prop. 3.7]{Audoux} to define and study the inverse map to Satoh's tube map, called the {\em connection map}.

\subsection*{Acknowledgements} The authors are grateful to Benjamin Audoux, Dror Bar-Natan and Hans Boden for several insightful conversations. We thank the Knot at Lunch Student Group: Alec Elhindi, Grace Garden, Tamara Hogan, Damian Lin, James Morgan, Tilda Wilkinson-Finch, and Grace Yuan for their insights and support. We also thank Brigitte Podrasky for her contribution of many illustrations.

\section{Solid ribbon torus links and welded link diagrams} 
In this section we introduce solid ribbon torus links and welded link diagrams in detail. 

\subsection{Solid ribbon torus links} Throughout, we denote by $\T$ an oriented solid torus $B^2 \times S^1$.

\begin{definition}\label{ribbon singularity}
For an immersion  $\cup_{i\in\{1,\dots,n\}}\T_{i}\rightarrow \R^{4}$, a {\em ribbon singularity} is a flatly transverse disc in the image. The preimage is a disjoint union of two discs: of these, one disc is in the interior of one of the tori $\T_{i}$; the other has its interior in the interior of one of the tori $\T_{j}$, and its boundary is essentially embedded in $\partial\T_{j}$. (Possibly $i=j$.) If an immersion has only ribbon singularities, it is called a {\em ribbon immersion}. 
\end{definition}

\begin{definition}
For $n\geq 1$, an oriented {\em solid ribbon torus link} with $n$ components is a ribbon immersion of $n$ oriented solid tori $\sqcup_{i\in \{1,\dots,n\}}\T_i$ in $\R^4$. Each solid torus $\T=B^2\times S^1$ is equipped with a {\em directed core} $\{0\}\times S^1 \subseteq \T$: an oriented circle.
\end{definition}

Next, we define equivalence of solid ribbon torus links. Ambient isotopy is a very fine notion of equivalence, which diagrammatically corresponds to welded diagrams but no Reidemeister moves \cite[Prop. 3.7]{Audoux}. In other words, Reidemeister moves relate solid ribbon torus links which belong to distinct ambient isotopy classes. To achieve a notion of equivalence more in line with intuition -- diagrammatically corresponding to welded diagrams up to welded moves -- one needs to introduce {\em generalised ribbon isotopies}, which are allowed to move through finitely many moments of immersions with {\em generalised ribbon singularities} \cite[Sec. 3]{AM}.

\begin{definition}\label{generalised singularities for ribbons}
Given an immersion $\cup_{i\in \{1,\dots,n\}}\T_{i}\rightarrow \R^{4}$, we say that a connected component $D$ of the singular set, contained in an open four-ball $B\subseteq \R^4$, is a {\em generalised ribbon singularity} (see also \cref{ribbonsingularities}) if it belongs to one of the following types, given in terms of a a local system of coordinates $\varphi$ for $B\cong \R^4$: 
\begin{enumerate} 
    \item[(0)] Type 0 (strict ribbon singularity, as in Definition~\ref{ribbon singularity}), where there is a local system of coordinates $\varphi$ for $B\cong \R^4$, such that $\varphi^{-1}(B)$ is a disjoint union $B_{1}\sqcup B_{2}$ of two 3-balls, and
    
    $\begin{dcases}
    \varphi(B_1)=\{(t,rcos(s),rsin(s),0)\} & t,s\in \R, r\in[0,2] \\
    \varphi(B_2)=\{(0,rcos(s),rsin(s),t)\} & t,s\in \R, r\in [0,1]
    \end{dcases}$
    
    \item[(2)] Type 2 (tangential, see bottom left of Figure~\ref{ribbonsingularities}), where there is a local system of coordinates $\varphi$ for $B\cong \R^4$ such that 
 $\varphi^{-1}(B)$ is the disjoint union $B_{1}\sqcup B_{2}$ of two 3-balls, and     
 
    $\begin{dcases}
    \varphi(B_1)=\{(0,rcos(s),rsin(s),t)\} & t,s\in \R, r\in[0,2] \\
    \varphi(B_2)=\{(-t^2,rcos(s),rsin(s),t)\} & t,s\in \R, r\in [0,1]
    \end{dcases}$
    
    \item[(3)] Type 3 (nested, see bottom right of Figure~\ref{ribbonsingularities}), where there is a local system of coordinates $\varphi$ for $B\cong \R^4$ such that  $\varphi^{-1}(B)$ is the disjoint union $B_{1}\sqcup B_{2}\sqcup B_{3}$ of three 3-balls, and 
    
    $\begin{dcases}
    \varphi(B_1)=\{(0,rcos(s),rsin(s),t)\} & t,s\in \R, r\in [0,3] \\
    \varphi(B_2)=\{(t,rcos(s),rsin(s),t)\} & t,s\in \R, r\in [0,2] \\
    \varphi(B_3)=\{(2t,rcos(s),rsin(s),t)\} & t,s\in \R, r\in [0,1]
    \end{dcases}$
\end{enumerate} 
\end{definition}

\begin{figure}
    \centering
    \includegraphics{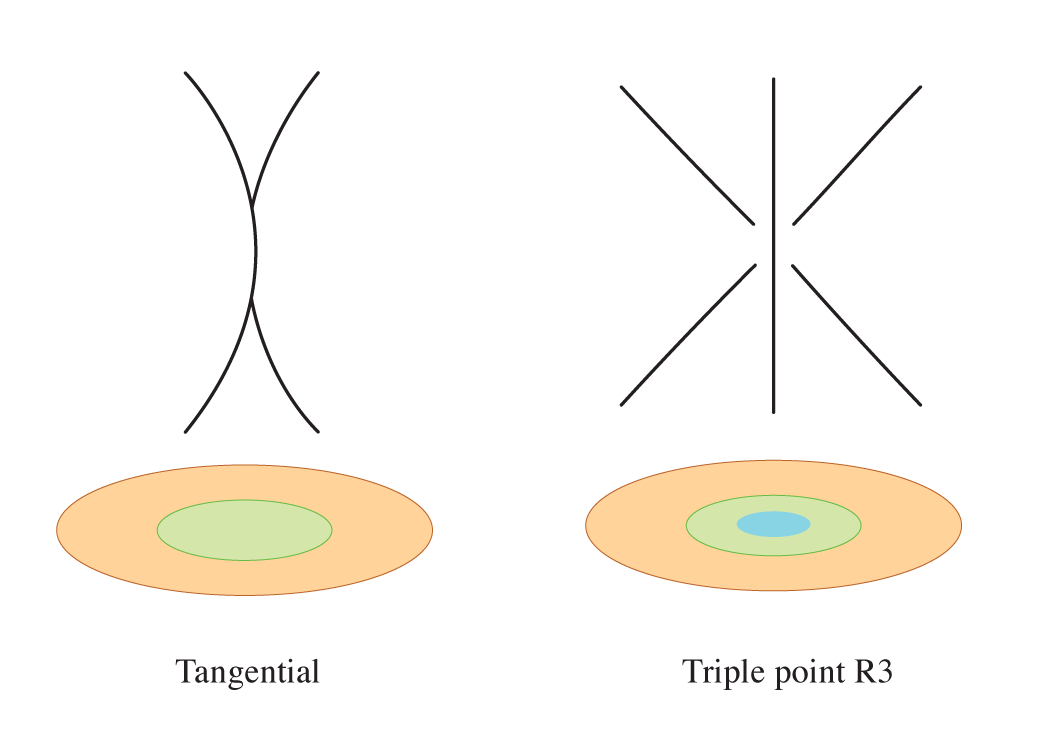}
    \caption{Type 2 and Type 3 generalised ribbon singularities in the preimage.}
    \label{ribbonsingularities}
\end{figure}

\begin{remark}
We will see later that the Type 2 and 3 singularities correspond to Reidemeister 2 and 3 moves.
There is no ``Type 1'' singularity, since a Reidemeister 1 move merely creates a cusp, not a new type of singularity.
\end{remark}

\begin{definition}
A {\em generalised ribbon isotopy} of solid ribbon torus links is a regular homotopy via ribbon immersions which may pass through finitely many moments of generalised ribbon immersions. 
\end{definition}

\subsection{Welded diagrams and the Tube map}
This note proves that solid ribbon torus links modulo generalised ribbon isotopy
are represented by welded diagrams modulo welded Reidemeister moves. We now define welded diagrams, and the enhanced Tube map, which relates them to solid ribbon torus links.

\begin{definition}\label{def:WeldedDiag}
An {\em oriented welded link diagram} (or {\em welded diagram} for short) is an immersion of a disjoint union of oriented circles in $\R^{2}$, with only transverse double points, each of which belongs to one of two types: classical crossings -- which may be positive $\neovnwarrow$ or negative $\nwovnearrow$ -- and virtual crossings $\tona$. Welded diagrams are considered modulo planar isotopy as four-valent graphs; the welded Reidemeister moves shown in \cref{Reid_moves}; and the convention that a purely virtual path (a path which goes through only virtual crossings) can be equivalently re-routed in any other purely virtual way (see \cref{fig:Rerouting}). 
\end{definition}

\begin{figure}
    \centering
    \includegraphics[width=14cm]{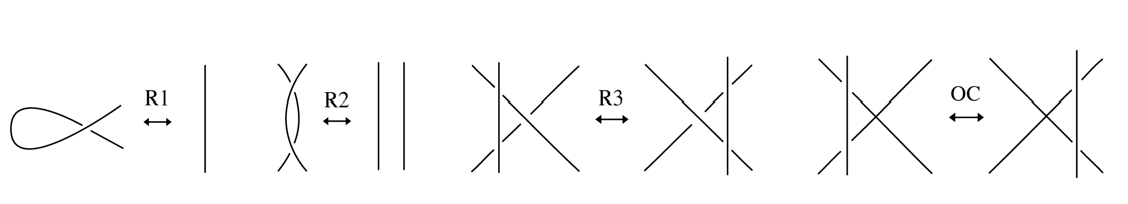}
    \caption{The Reidemeister moves for welded diagrams. Orientations are suppressed: the relations are imposed in all consistent orientations of the strands.}
    \label{Reid_moves}
\end{figure}

\begin{figure}
\centering
 \includegraphics[width=7cm]{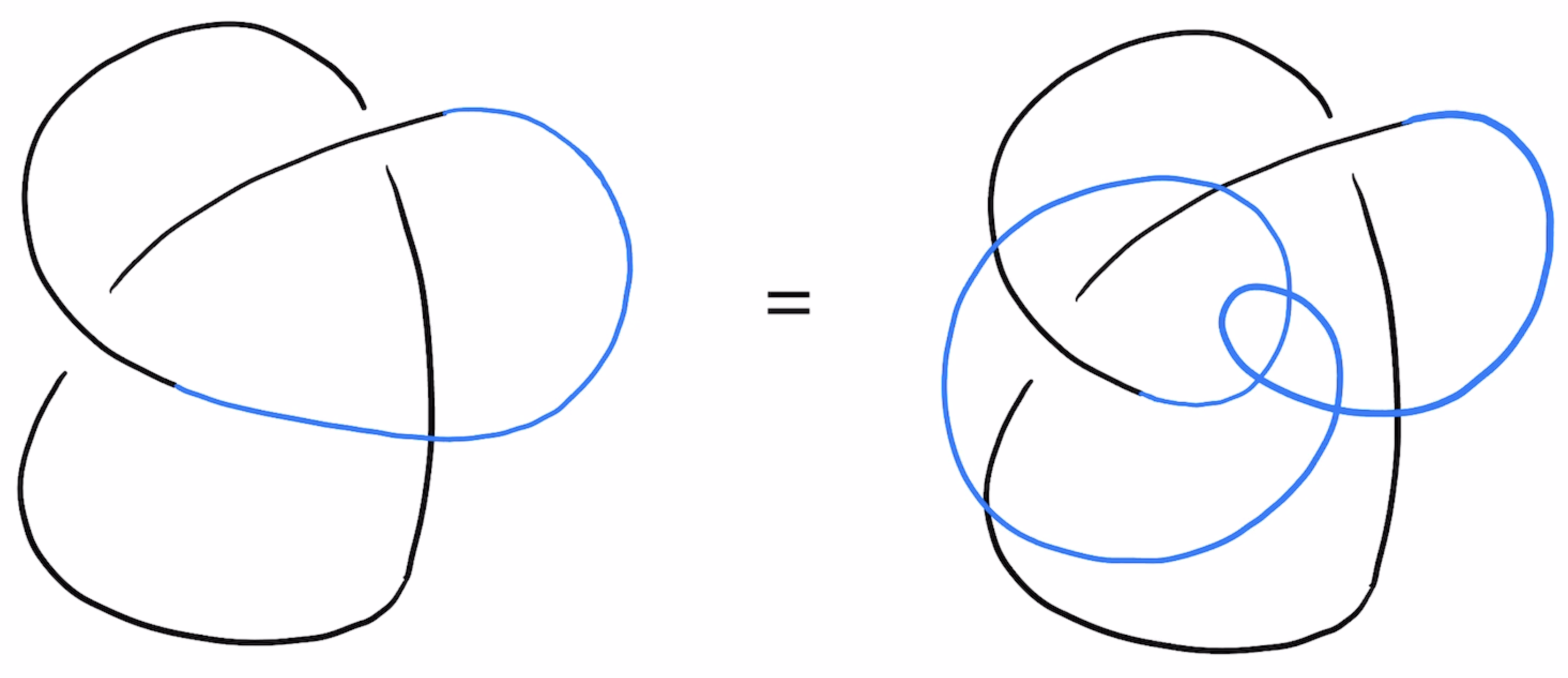}
    \caption{An example of the rerouting move for a purely virtual strand.}
    \label{fig:Rerouting}
\end{figure}

\begin{remark}
 The re-routing convention can be replaced with a set of three {\em virtual Reidemeister} moves and one {\em mixed} move. We prefer the re-routing convention, which is a step closer to Gauss codes. Gauss codes encode welded diagrams as a list of classical crossings, and the list of connections between their endpoints, ignoring the planar embedding (and virtual crossings) of the strands which create these connections. 
\end{remark}

There are three maps which clarify the relationship between solid ribbon torus links and welded diagrams. The enhanced spun and tube maps \cite{Satoh} construct solid ribbon torus links from diagrams: the enhanced spun map from classical diagrams (with no virtual crossings), and the enhanced tube map  from all welded link diagrams. The connection map \cite{Audoux} constructs a welded diagram from a solid ribbon torus link. We will prove that for solid ribbon torus links up to generalised ribbon isotopy, the enhanced tube map and the connection map are mutual inverses.

\medskip
The most intuitive of the three maps is the {\it enhanced Spun map} for link diagrams with only classical crossings. Such a diagram represents a link in $\R^3$: that is, an embedding $\cup_{i\in\{1,...,n\}}S^1 \hookrightarrow \R^3$. The enhanced Spun map, denoted $Spun^\bullet$, associates to this classical link a solid ribbon torus link, by including $\R^3$ in $\R^4$ and spinning the link in $\mathbb{R}^4$ around a regular projection plane. The output is the surface ``drawn'' by the spun link, filled by the union of projection rays: this indeed is a solid ribbon torus link with one ribbon singularity for each crossing in the link diagram.

The {\em enhanced tube map} generalises this construction to all welded diagrams: 
$$\text{Tube}^{\bullet}:\{\text{welded link diagrams}\}\rightarrow \{\text{solid ribbon torus links}\}$$ 

Applying $\text{Spun}^{\bullet}$ to only a disc neighbourhood of a classical crossing results in two cylinders $B^2\times I$, one of which is threaded through the other in $\R^4$: call this a spun crossing. 
Given a welded diagram with $r$ classical crossings, for each classical crossing $\text{Tube}^{\bullet}$ embeds a spun crossing in disjoint 4-balls $B^4$ in $\R^4$, in which the two bottom and two top discs $B^2\times \{0\}$ and $B^2\times \{1\}$ of the spun crossing are on the boundaries $\partial B^4$: call these {\em boundary discs}.  
Denote by $Y$ the complement of the $r$ 4-balls which contain the spun crossings. In the diagram there are strands (plane-immersed intervals) connecting the ``ends'' of the crossings. For each such embedded interval, glue in a regularly embedded $D^{2}\times I$ in $Y$ to the appropriate boundary discs in $\partial Y$.

\begin{prop}\label{well-defined}
Up to strict ribbon isotopy, $\text{Tube}^{\bullet}$ is well-defined. 
\end{prop}

The proof we present is a slightly extended version of the corresponding part of \cite[proof of Proposition 3.7]{Audoux}:

\begin{proof} In the $\text{Tube}^{\bullet}$ construction, choices of specific embeddings are made when gluing in the connecting solid tubes $D^2 \times I$. We need to show that such embeddings always exist and are equivalent up to isotopy.

Indeed, since each $D^2\times I$ is embedded in $Y$, it can be retracted to a path  connecting $D^2 \times \{0\}$ to $D^2 \times \{1\}$. The embedding of $D^2\times I$ specifies a {\em framing}\footnote{A framing for a path is a section of the normal bundle, that is, a continuous choice of normal vectors along the path.} for this path: namely, choose the positive normal vector to $D^2\times I$ in $\R^4$ at each point. Call this path a {\em framed core} for the embedded cylinder $D^2 \times I$. 

On the other hand, a framed path $\gamma$ between the boundary 2-balls can be ``inflated'' to a regularly embedded 3-ball $D^2 \times I$ connecting these 2-balls. To do so, for each point $x\in \gamma$, consider a 2-dimensional neighbourhood of $x$ which is orthogonal to both the derivative of $\gamma$ and the framing, and take the union of a continuous choice of such neighbourhoods.

These retractions and inflations are mutual inverses up to isotopy of, on one hand, the embedded 3-balls fixing $D^2 \times \{0\}$ and $D^2 \times \{1\}$; and on the other hand the framed paths, fixing the endpoints. Thus, two embeddings of $D^2\times I$ between the same boundary 2-balls are isotopic if and only if their framed cores are. The framing along a path $\gamma$ is always orthogonal to $\gamma$, so it can be seen as a path on $S^2$, and since $\pi_1(S^2)=1$, all framings are equivalent up to isotopy. Thus, it is enough to consider isotopy of unframed paths. Finally, since the paths 1-dimensional, all embeddings in the punctured 4-ball $Y$ (with the given starting and ending points) are isotopic.
\end{proof}

\begin{prop}\label{R-moves rep isotopies}
Under the $\Tube^\bullet$ map, the two sides of each welded Reidemeister move are equivalent under generalised ribbon isotopy.
\end{prop}
\begin{proof}
The Reidemeister 1, 2, and 3 moves of welded link diagrams are local equivalences between locally classical string link diagrams. Thus, $Tube^\bullet$ and $Spun^\bullet$ coincide for these moves, and applying $Spun^\bullet$ to the corresponding classical link isotopies gives generalised ribbon isotopies of solid ribbon torus links. In fact, R1 gives a strict isotopy, R2 a generalised ribbon isotopy with a single type 2 singularity, and R3 a generalised ribbon isotopy with a single type 3 singularity.  

The OC move involves a virtual crossing, so it does not follow from classical isotopies. We represent the $Tube^\bullet$ image of both sides as movies of flying discs as shown in \cref{fig: OC-FlyingCircles}: On the left, disc one (blue) flies through disc two (orange), and then disc three (red) flies through disc two (orange). On the right, disc three (red) flies through disc two (orange) first, then disc one (blue) flies through disc two. These movies are isotopic via the middle movie of discs one and three flying through disc two side by side at the same time. This is a strict ribbon isotopy of solid ribbon torus links.
\end{proof}

\begin{figure}
    \centering
    \includegraphics[width=11cm, height=8cm]{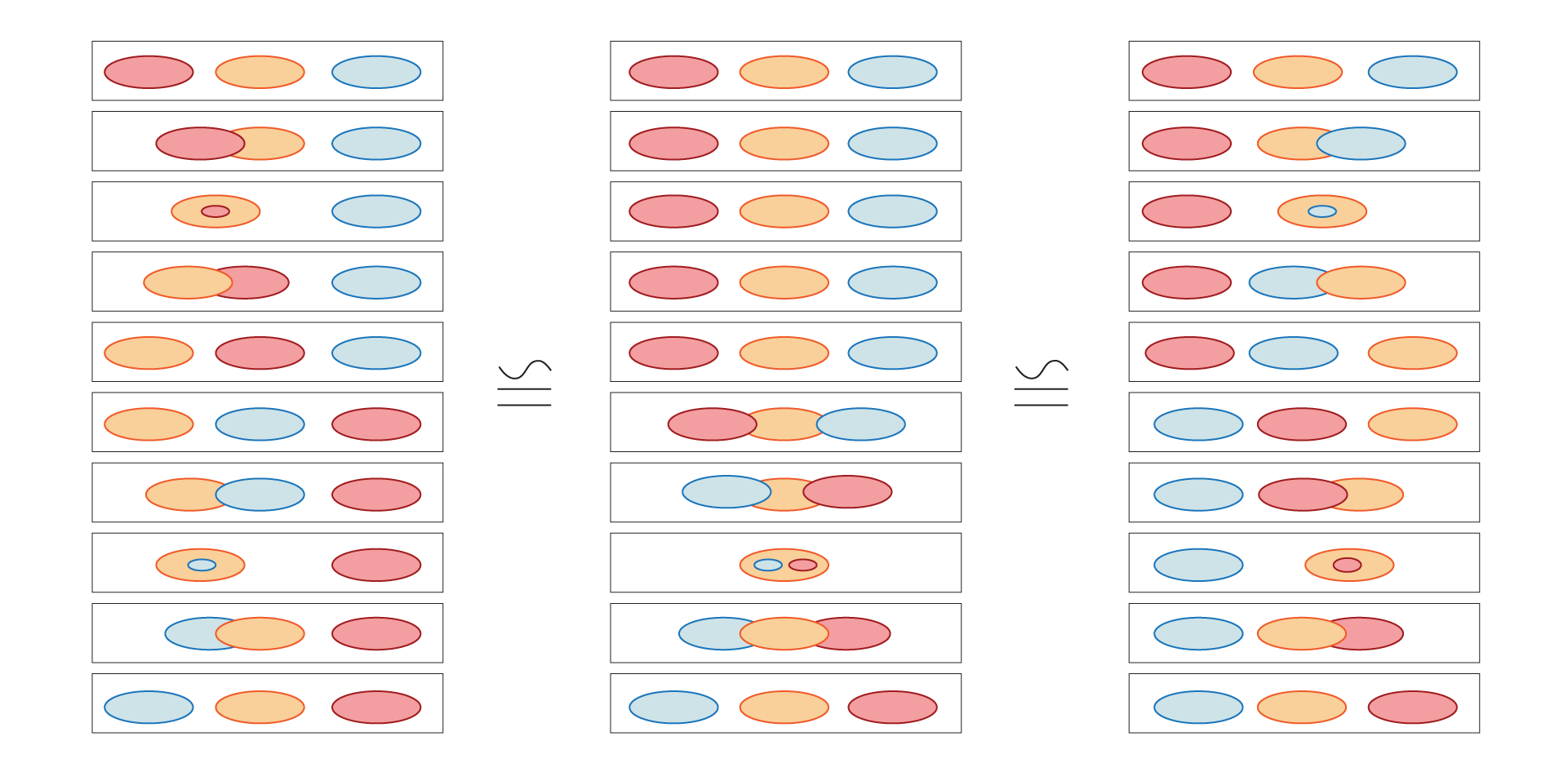}
    \caption{The $Tube^\bullet$ image of the OC move, represented as a homotopy of movies of flying discs.}
     \label{fig: OC-FlyingCircles}
\end{figure}

\begin{cor}
The enhanced tube map descends to a well-defined map on generalised welded link diagrams modulo welded Reidemeister moves. 
\end{cor}

 The goal of this section is to prove \cref{thm:main}, namely, to show that there exists a bijection between solid ribbon torus links (up to generalised ribbon isotopy), and welded link diagrams up to welded equivalence. In order to do so we introduce the Connection ($\Conn$) map inspired by \cite[Section 3.2]{Audoux}, and show that it is inverse to $Tube^\bullet$. In particular, the $\Conn$ map assigns a welded diagram to a solid ribbon torus link: crossings are associated to ribbon singularities and connected based on the sequence of preimages of singularities within each torus. The following theorem is a more precise version of Theorem~\ref{thm:main}:
 
\begin{theorem}\label{Thm: Conn} The connection map is a set bijection 
$$\Conn: \frac{\{\text{solid ribbon torus links}\}}{\text{generalised ribbon isotopy}}\rightarrow \frac{\{\text{welded diagrams}\}}{\text{welded  equivalence}}.$$ 
\end{theorem}

We define the Connection map for solid ribbon torus links (ignoring isotopies) first, in two separate parts:
\begin{enumerate}
    \item $\sigma_{\Conn}(R)$ describes the {\it sign} of each crossing using the orientation of the tori around each singularity.
     \item $t_{\Conn}(R)$ prescribes how to connect the  crossings in the welded diagrams, from the combinatorial structure of the singularities of the solid ribbon torus links
\end{enumerate}

For a ribbon immersion $R:\cup_{i\in \{1,...,n\}} \mathbb{T}_i \to \mathbb{R}^4$,
let $\Delta=\{\delta_1, \delta_2, \dots, \delta_k\}$ denote the set of ribbon singularities. These correspond to the set of (classical) crossings in the welded diagram.

To each ribbon singularity $\delta$ we associate a sign. Let $\gamma\subseteq \mathbb{T}_k$ and $\varepsilon \subseteq \mathbb{T}_l$ denote, respectively, the contractible and essential preimages of $\delta$. Let $x$ be a point in the interior of $\delta$, and let $x_c\in \gamma$ and $x_e\in \varepsilon$ denote the preimages of $x$. Choose a positive basis in $T_{x_c} \mathbb{T}_k$, and denote by $\{u_1, u_2, u_3\}$ its push-forward in $T_x (\mathbb{R}^4)$. Let $v \in T_x(\mathbb{R}^4)$ be a push-forward of a normal vector to $T_{x_e}(\varepsilon) \subseteq T_{x_e} \mathbb{T} _l$, pointing ahead of  $\varepsilon$ is the direction of the core. Then the sign of $\delta$ is defined to be the sign of the basis $u_1, u_2,..., u_3, v$ for $T_x(\mathbb{R}^4)$. 

\begin{definition}[$\sigma_{\Conn(R)}$]
The map
$$\sigma_{\Conn(R)}:\Delta \rightarrow \{-1,1\} $$
assigns to each ribbon singularity its sign.
\end{definition}

Thus, from a solid ribbon torus link, we have produced a list of oriented crossings. Place these oriented crossings arbitrarily in disjoint discs in the plane (see the red, green and blue discs in \cref{fig: Conn Map Example}). Here {\em arbitrarily} means that the location, direction (strands pointing up, down, left or right), and size of the crossings are all chosen arbitrarily, as long as the crossings occupy disjoint discs, numbered, and are of the correct sign. What remains is to describe how to connect these crossings to produce a welded diagram.

\begin{figure}
    \centering
    \includegraphics{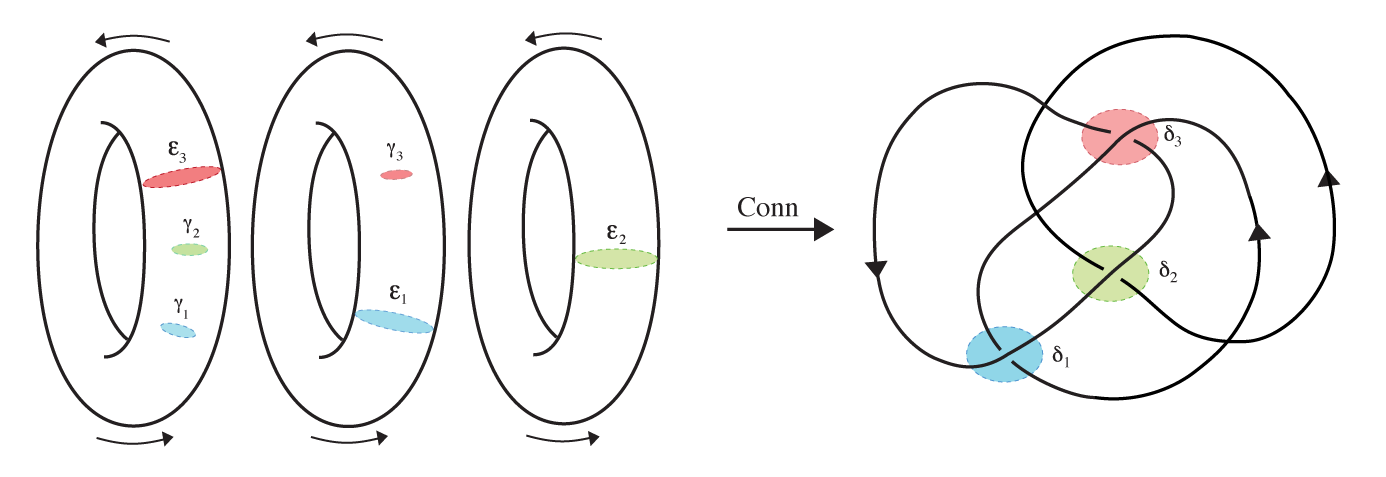}
    \caption{On the left are the preimages of the ribbon singularities. On the right is the welded diagram: the output of the Connection map.}
     \label{fig: Conn Map Example}
\end{figure}

Denote the set of essential preimages of ribbon singularities by $E=\{\varepsilon_1,...,\varepsilon_n\}$ and the contractible preimages by $\Gamma=\{\gamma_1,...,\gamma_k\}$.
The set $E \cup \Gamma$ is partitioned into $n$ subsets, according to which of the $n$ tori each preimage falls in:
$$E\cup \Gamma = \sqcup_{i=1}^n S_i.$$
For example, in \cref{fig: Conn Map Example}, $S_1=\{\gamma_1, \gamma_2, \varepsilon_3\}, S_2=\{\varepsilon_1, \gamma_3\}$, and $S_3=\{\varepsilon_2\}$.

\begin{definition}[$t_{\Conn(R)}$] The component $t_{\Conn(R)}$ is a partial cyclic ordering on each of the partition sets $S_i$.
Each essential preimage intersects the core of its component $\mathbb{T}_i$ at least once. Since the core is oriented, the order of first intersections is a complete cyclic order on the essential singularities in $S_i$.
Furthermore, the essential singularities in $\mathbb{T}_i$ partition $\mathbb{T}_i$ into chambers. Each contractible singularity in $\mathbb{T}_i$ belongs to a unique chamber, and thus is cyclically comparable to the essential singularities in $\mathbb{T}_i$, and to the contractible singularities in other chambers. Contractible singularities in the same chamber are not comparable to each other.
\end{definition}

This partial cyclic order on the $\{S_i\}_{i=1}^n$ determines the welded diagram, as follows.  Each essential singularity $\varepsilon_r \in S_i$ corresponds to the under-strand of the crossing $r$. Identify the next essential singularity $\varepsilon_s \in S_i$ in the partial cyclic order $t_{\Conn(R)}$ on $S_i$, noting that possibly $r=s$. There is possibly a set of contractible singularities $\{\gamma_{p_1},...,\gamma_{p_m}\}$ between $\varepsilon_r$ and $\varepsilon_s$ in the partial order of $S_i$. Connect the outgoing under-strand of crossing $r$ to the incoming under-strand of crossing $s$ via the over-strands of crossings $p_1, \ldots ,p_m$, introducing virtual crossings if necessary. Note that different orderings of going over the crossings $p_1, \ldots ,p_m$ are all OC-equivalent (as in the middle tube of \cref{OC}). Continue this process until all crossings have been connected.
A special case is when a torus contains only contractible singularities, corresponding to over-strands only. In this case, connect those over-strands in a loop, in any order, introducing virtual crossings as needed.

\begin{figure}
    \centering
    \includegraphics[height=8cm]{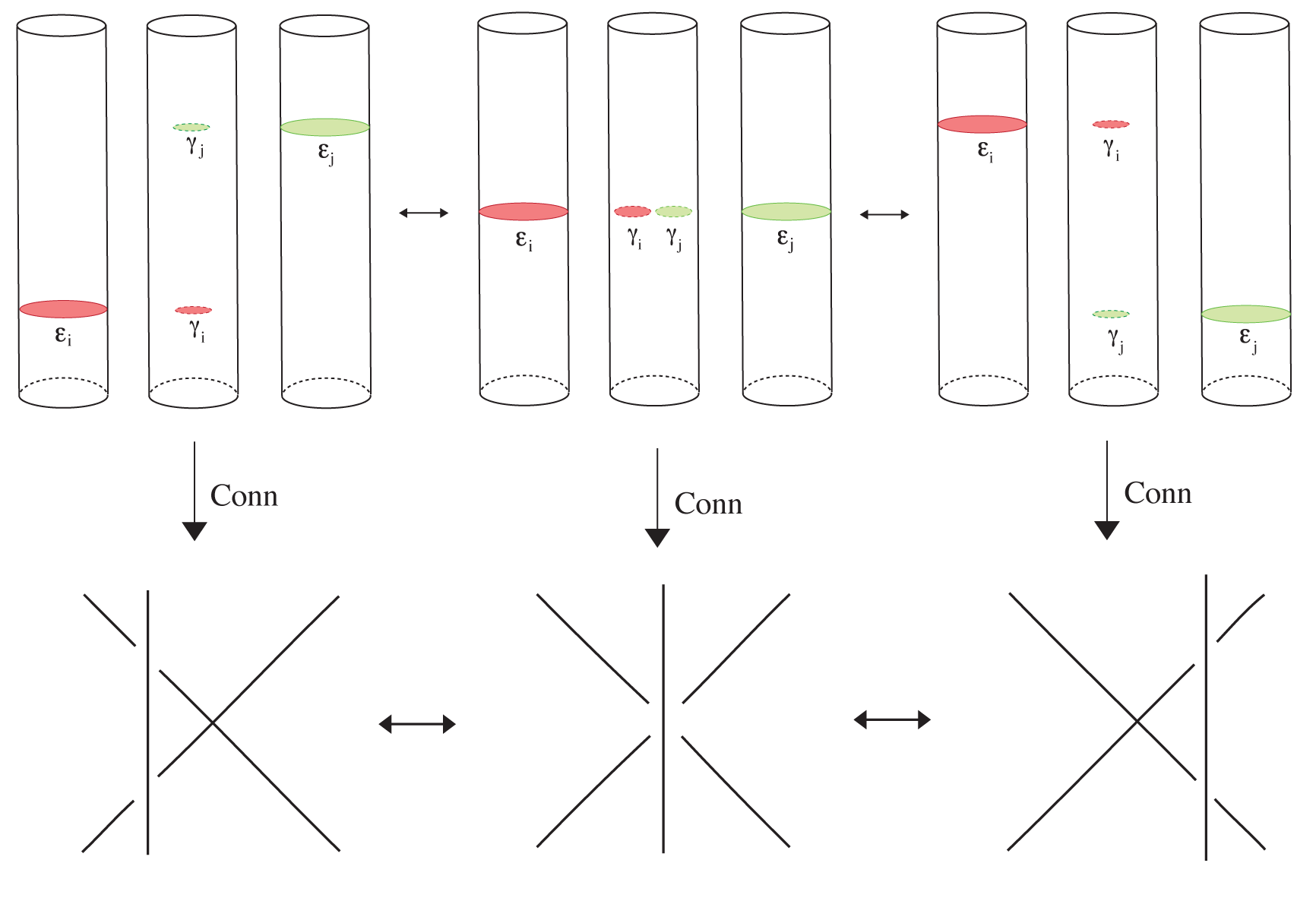}
    \caption{Differently ordered connections of over-strands corresponding to contractible singularities within a chamber are OC-equivalent.}
    \label{OC}
\end{figure}

\begin{example}
The left hand side of \cref{fig: Conn Map Example} shows the preimages of ribbon singularities for a solid ribbon torus link: for $i=1,2,3$, the contractible preimage of $\delta_i$ is labeled $\gamma_i$, and the essential preimage is labelled $\varepsilon_i$. 

The $\Conn$ map provides the information for constructing the welded diagram shown on the right of \cref{fig: Conn Map Example}. We know that there are three singularities, corresponding to three crossings, and $\sigma_{\Conn(R)}(\delta_i)$ specifies their signs. We place the three crossings on the right (in any position), highlighted red, blue, and green. 

Based on the first torus, we need to connect the outgoing under-strand of $\delta_3$ (red) to the incoming under-strand of the same crossing, and on the way cross over $\delta_1$ and $\delta_2$ in any order. Then, based on the second torus, we connect the outgoing under-strand of $\delta_1$ (blue) to the incoming under-strand of the same crossing, on the way crossing over $\delta_3$ (green). This loop introduces a virtual crossing. Finally, based on the third torus we connect the outgoing under-strand of $\delta_2$ (green) to the incoming under-strand of the same crossing, introducing some virtual crossings along the way.
\end{example}

Working towards the proof of \cref{Thm: Conn}, we first prove that the $\Conn$ map -- which {\em a priori} maps solid ribbon immersions to welded diagrams modulo the OC move -- descends to a well-defined map of solid ribbon torus links up to generalised ribbon isotopy, to welded diagrams modulo welded Reidemeister moves. 

\begin{theorem}\label{Isotopy=Reid}
If $R_0$ and $R_1$ are solid ribbon torus links related by generalised ribbon isotopy, then $\Conn(R_0)$ and $\Conn(R_1)$ are related by Reidemeister moves of welded link diagrams.
\end{theorem}

\begin{proof}
The image of the $\Conn$ map depends only on the combinatorial structure of the preimages of singularities, and the orientations around the ribbon singularities of a solid ribbon torus link. Hence, we need to analyse how this combinatorial structure changes throughout a generalised ribbon isotopy. 

Let $\eta: (\cup_i (B^2\times I)_i)\times I \to B^3\times I$ be a generalised ribbon isotopy of ribbon torus links, and let $\eta_t$ denote the immersion $\eta_t: (\cup_i (B^2\times I)_i)\times \{t\} \to B^3\times I$. Assume that $\eta_0 \equiv R_0$ and $\eta_1 \equiv R_1$. Recall at finitely many discrete values of $t$, $\eta_t$ may include generalised ribbon singularities, as in \cref{generalised singularities for ribbons}.

In the preimage, $\eta$ manifests as a movie of the essential and contractible preimages of the singularities moving around within the tori, and appearing or disappearing at discrete times. With the exception of finitely many discrete values of $t$, this does not change the combinatorial structure (partial ordering of essential singularities or signs), and hence it does not affect the value of the $\Conn$ map. The goal is to analyse the discrete values of $t$ where the value of the $\Conn$ map does change. We base the case analysis on whether a generalised singularity occurs at $t$, and which kind.  

\medskip 

\noindent {\em Case 1: No Generalised Singularities.} In the preimage of a solid ribbon torus link, the preimages of the singularities do not intersect. Therefore, an isotopy which does not involve generalised singularities cannot change the cyclic ordering of the essential preimages of existing essential singularities, and the signs of the singularities are also rigid. However, the value of the Connection map may change, as singularities may appear or disappear. Since there are no tangential intersections, each appearing/disappearing singularity involves only one tube, and the preimages must be adjacent. By local deformation one may arrange for such appearances/disappearances to occur one at a time. A single instance of this is shown in \cref{Reid1}. Observe that the effect of this local isotopy on the value of the Conn map is an R1 move.

\begin{figure}
    \centering
    \includegraphics[width=13cm]{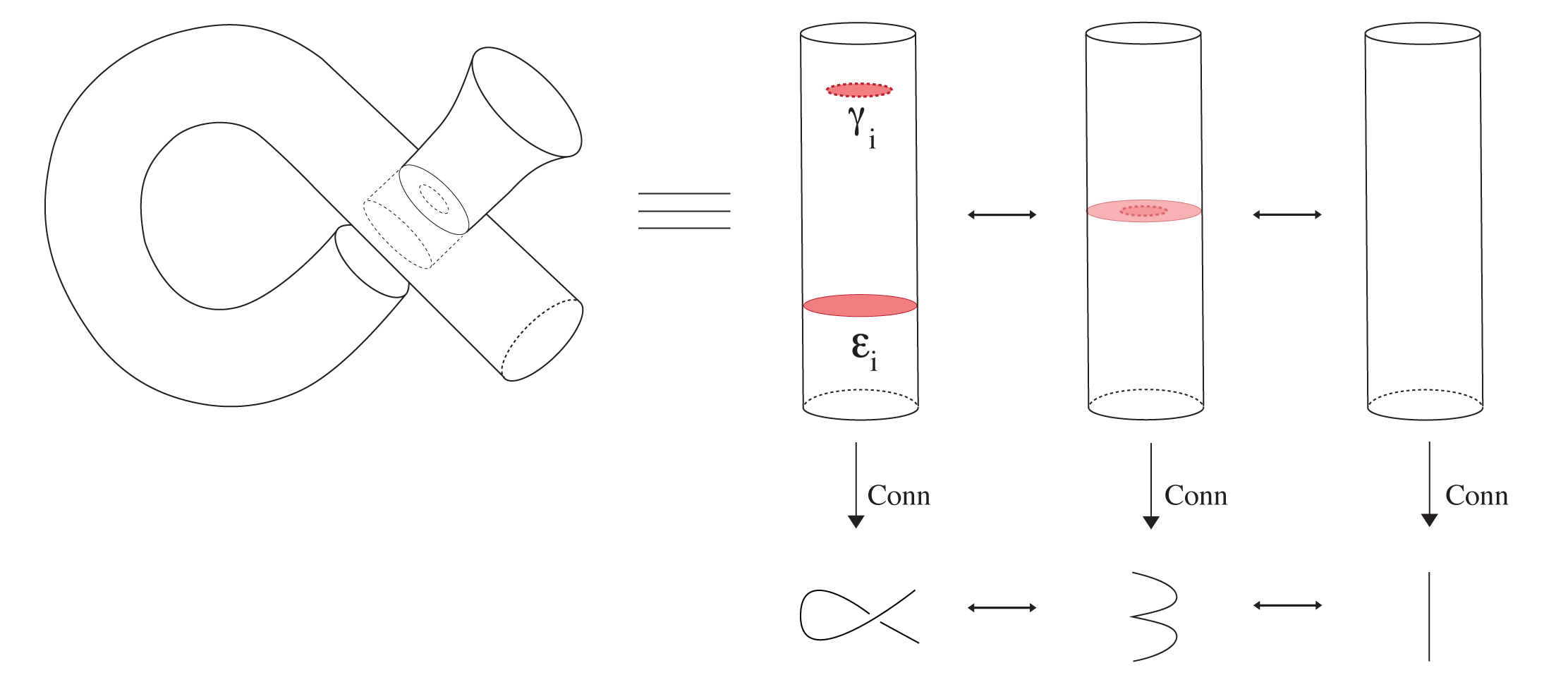}
    \caption{An isotopy creating a single ribbon singularity via a cusp.}
    \label{Reid1}
\end{figure} 

\medskip
\noindent {\em Case 2: Type 2 generalised ribbon singularities.}
A type 2 (tangential) generalised ribbon singularity can be eliminated by a small local perturbation two ways: one resulting in no singularity, the other in two (transverse) ribbon singularities. In the image of the Connection map, this amounts to a Reidemeister 2 move: see \cref{Reid2}.

\begin{figure}
    \centering
    \includegraphics[height=8cm]{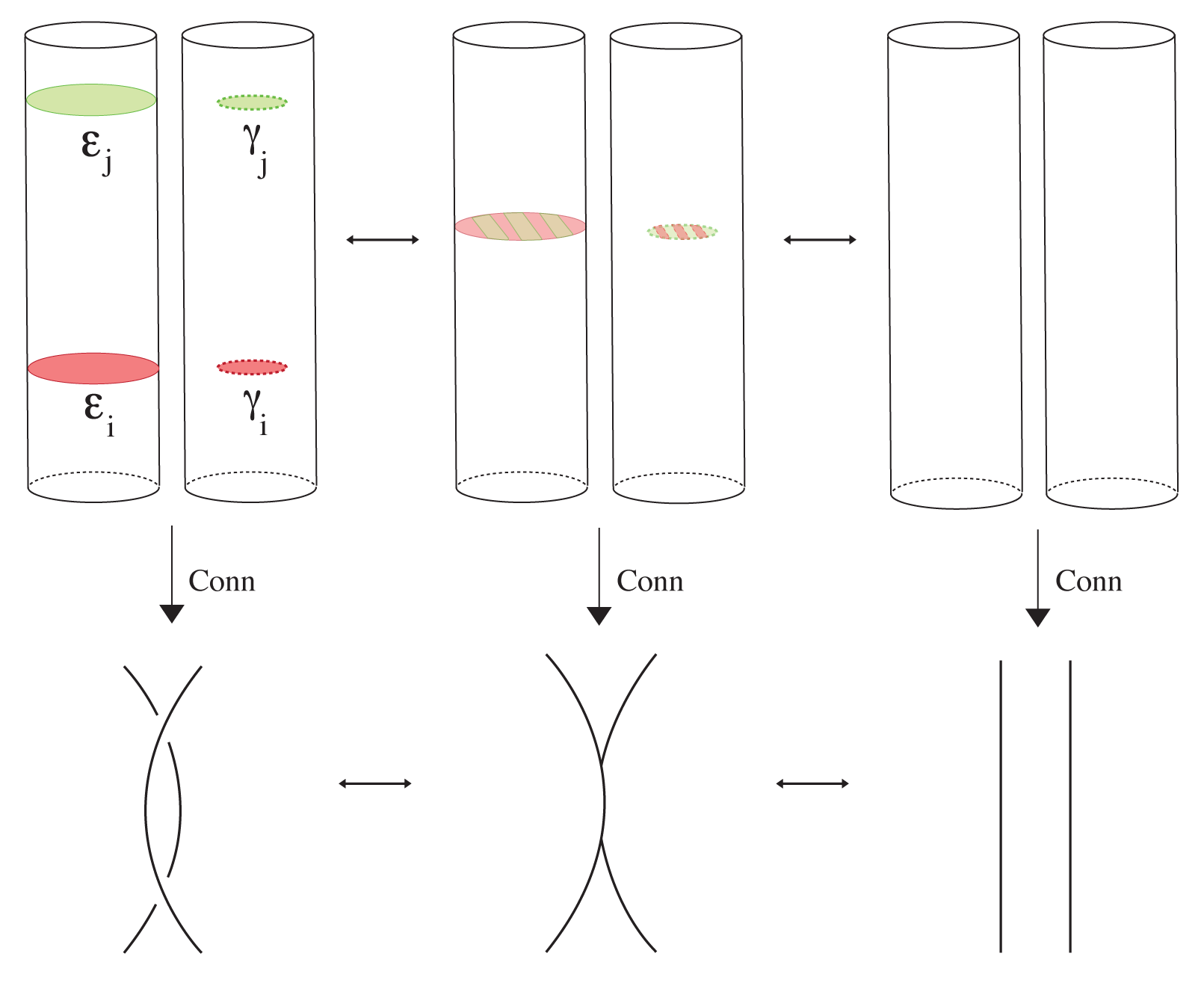}
    \caption{Two possible resolutions of a Type 2 (tangential) generalised ribbon singularity.}
    \label{Reid2}
\end{figure} 

\medskip
\noindent {\em Case 3: Type 3 generalised ribbon singularities}. The preimage of a nested singularity is shown in \cref{Reid3}, along with the two possible ways of resolving it. Each of the resolutions results in three ribbon singularities. The difference between the two resolutions in the image of the Connection map is a Reidemeister 3 move.  
\begin{figure}
    \centering
    \includegraphics{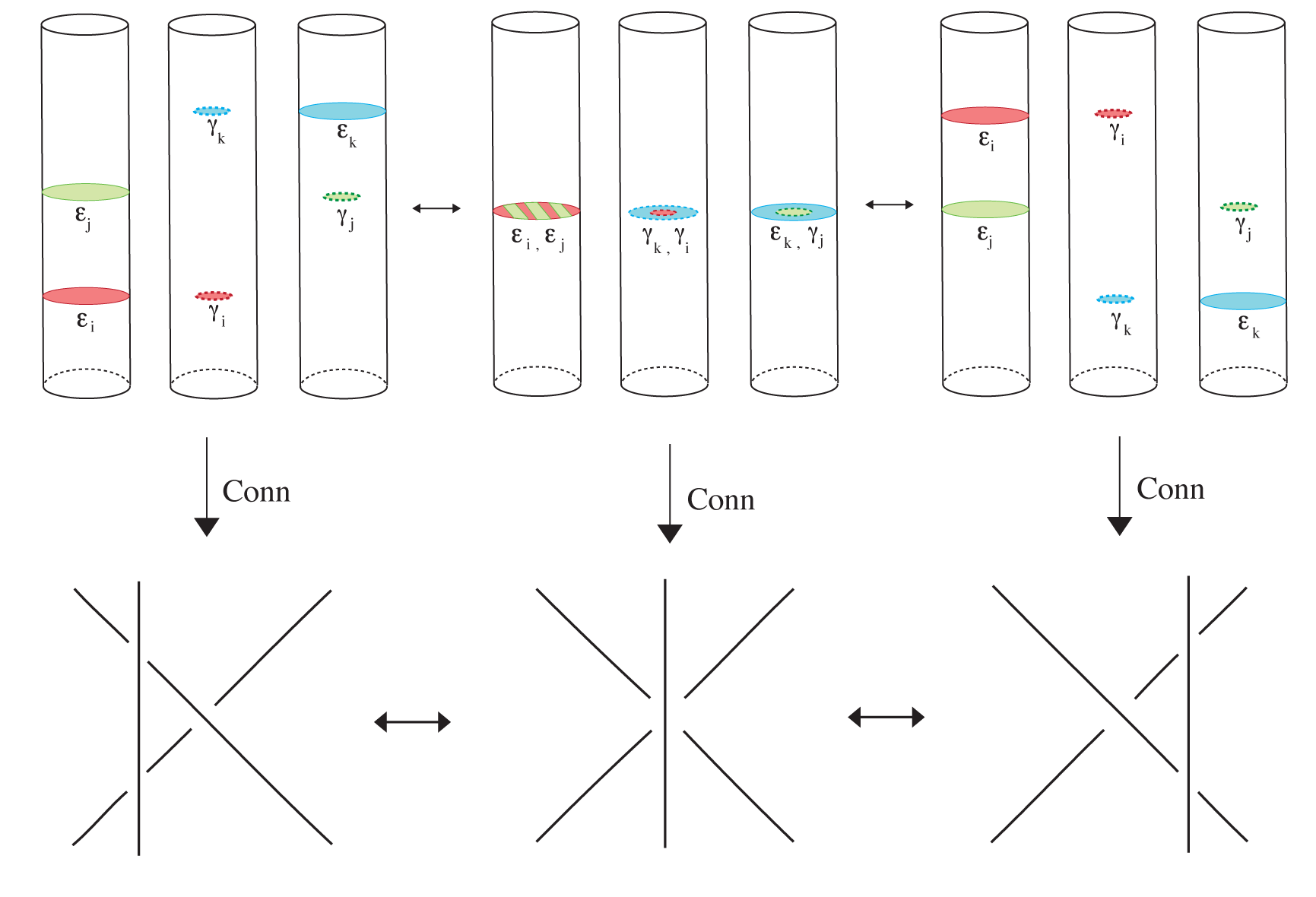}
    \caption{Two possible resolutions of a Type 3 (nested) generalised ribbon singularity.}
    \label{Reid3}
\end{figure}

Hence, isotopies of solid ribbon torus links manifest as Reidemeister moves in the image of the Connection map. 
\end{proof}

\begin{cor}\label{cor:ConnWellDef}
The map $\Conn$ on solid ribbon torus links descends to a well-defined map on solid ribbon torus links up to generalised ribbon isotopies $$\Conn: \frac{\{solid \ ribbon \ torus\ links\}}{generalised \ ribbon \ isotopy}\rightarrow \frac{\{welded \ diagrams\}}{welded \ equivalence}.$$ 
\end{cor}

With \cref{Isotopy=Reid} and \cref{cor:ConnWellDef} in place, the rest of the proof of the main Theorem uses the ideas of \cite[Prop. 3.7]{Audoux}.

\begin{proof}[Proof of \cref{Thm: Conn}]
We need to show that the $\Conn$ map is surjective and injective.
Surjectivity is immediate from \cref{well-defined}, as $\Tube^\bullet$ is a right-sided inverse to $\Conn$: given a welded diagram $D$, apply the $\Tube^\bullet$ map to produce a solid ribbon torus link $\Tube^\bullet(D)$. Then by definition, $\Conn(\Tube^\bullet(D))=D$.

For injectivity, consider two solid ribbon torus links $R_1$ and $R_2$ such that $\Conn(R_1)\sim \Conn(R_2)$; we need to show that $R_1$ is generalised ribbon isotopic to $R_2$. By \cref{R-moves rep isotopies}, Reidemeister moves represent local generalised ribbon isotopies under the $\Tube^\bullet$ map, hence we may assume that $\Conn(R_1)= \Conn(R_2)$.

The equality $\Conn(R_1)=\Conn(R_2)$ induces a one-to-one correspondence between the ribbon singularities
\[\Delta_{R_1}\xleftrightarrow{\xi}\Delta_{R_2}.
\] 
Number the crossings of $\Conn(R_1)$ and use the bijection $\xi$ to number the crossings of $\Conn(R_2)$ in the same order.

By local isotopy, the essential singularities can be deformed to essential horizontal discs:
The boundary of an essential singularity is an embedded ${S}^1$ along the boundary of the tube, and thus can be transformed by local isotopy to a horizontal circle. Then we use the standard argument to deform the interior of the disc to a horizontal disc: take a tubular neighbourhood of the singularity, and
\begin{enumerate}
    \item If the essential preimage does not intersect the horizontal disc filling its boundary, then their union is a 2-sphere, which admits a filling and hence an isotopy of the essential preimage to the flat disc.
    \item If the interior of the essential preimage intersects the interior of the flat disc filling the boundary ${S}^1$, first ensure by local isotopy that all intersections are transverse double points, see on the left of \cref{inner_loops1}. Thus, the intersections form a disjoint union of finitely many nested closed loops, as on the right of \cref{inner_loops1}. Select a loop innermost in the nesting, this loop can be removed by filling the sphere formed by its interior disc and the flat disc it bounds. Thus, all intersections can be recursively removed to reduce to the previous case.
    
\end{enumerate}
\begin{figure}
    \centering
    \includegraphics[height=5cm]{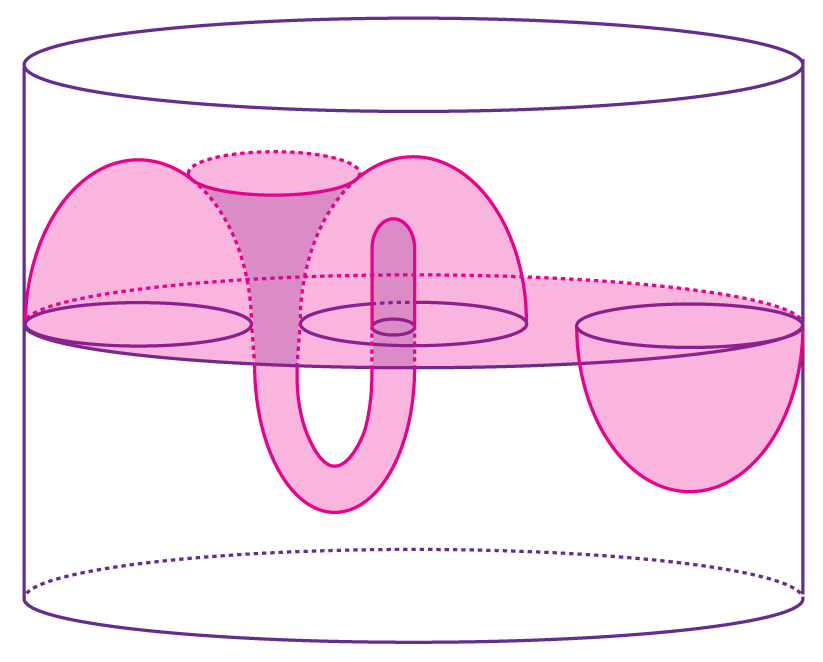}\hspace{1cm}
    \raisebox{5mm}{\includegraphics[height=4cm]{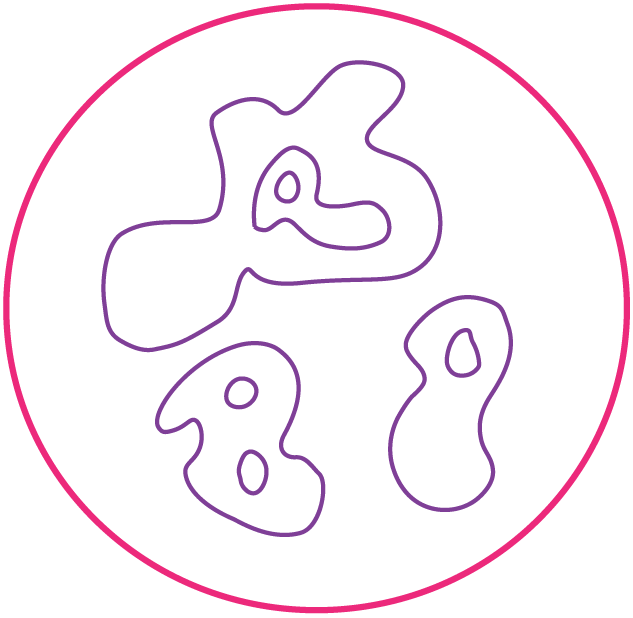}}
    \caption{Deforming the interior of an essential singularity.}
    \label{inner_loops1}
\end{figure}

By a further re-parametrisation, obtain $\tilde{R}_1$ and $\tilde{R}_2$ with the properties:
\begin{enumerate}
    \item Each contractible preimage of a ribbon singularity $\delta$ belongs to the interior of a slice $\left(B^2 \times \{t_{\delta}\} \right)_{i_{\delta}}$ of $\Tilde{R}_{1}$ or $\Tilde{R}_{2}$ for some $i_{\delta}\in \{1,2,\dots,n\}$ and $t_{\delta}\in I$;
    \item for every ribbon singularity $\delta$ of $\tilde{R}_1$, $\left(i_{\xi(\delta)}, t_{\xi(\delta)} \right)=\left(i_{\delta},t_{\delta}\right)$.
\end{enumerate}
The first point is achieved by Alexander's Trick \cite{Alexander}. The second point is achievable by re-parameterising because the partial ordering of the essential singularities are the same. 

For each ribbon singularity $\delta_i$ of $\tilde{R}_1$, fix a small tubular neighbourhood $U_i$ of its image: a 4-ball, which intersects $\tilde{R}_1$ in exactly four 2-balls. Then, since the signs and ordering of the singularities agree, $\tilde{R}_2$ can be isotoped such that $\tilde{R}_1 \cap U_i = \tilde{R}_2 \cap U_i$, that is, the singularities in the two links happen at exactly the same place in exactly the same way. The rest of the proof then amounts to isotoping the solid tubes which connect the singularities to match, which is the same as the proof of the well-definedness of the $\Tube^\bullet$ map (\cref{well-defined}).
\end{proof}

\end{document}